\documentclass[10pt]{article}
\usepackage{tikz}
\usepackage{amsmath}
\usepackage{amssymb}
\usepackage{amsthm}

\usepackage{enumerate}

\newtheorem{lemma}{Lemma}

\newtheorem{theorem}[lemma]{Theorem}
\newtheorem{corollary}[lemma]{Corollary}

\theoremstyle{definition}
\newtheorem{definition}{Definition}

\theoremstyle{remark}

\newtheorem{remark}{Remark}

\begin{document}
\begin{center}
	{\Large Twisted recurrence for dynamical systems with exponential decay of correlations}\\
	\vspace{.3cm}
	 Jiajie Zheng
	\end{center}
	
	\smallskip
	
	\begin{abstract}
		We study the set of points returning infinitely often to a sequence of targets dependent on the starting points. With an assumption of decay of correlations for $L^1$ against bounded variations, we prove a generalized quantitative recurrence result under Lipschitz twists.   
	\end{abstract}
	
\section{Introduction}

Let $(X,d)$ be a separable and compact metric space, and let $(X,\mu,T)$ be a Borel probability measure-preserving system. The Poincar\'{e} Recurrence Theorem, see e.g. \cite{EW} states that almost all points in measurable dynamical systems return close to themselves under a measure-preserving map; i.e.,  
\begin{equation*}
	\liminf_{n\to\infty}d(T^nx,x)=0 \quad \text{for almost every }x\in X.
\end{equation*}
Boshernitzan quantified the speed of recurrence in \cite{bosh}. Namely, if the $\alpha$-dimensional Hausdorff measure of $X$ is $\sigma$-finite for some $\alpha>0$, then 
\begin{equation*}
	\liminf_{n\to\infty}n^{1/\alpha}d(T^nx,x)<\infty \quad \text{for almost every }x\in X.
\end{equation*}
A natural generalization of the recurrence speed is to consider the following set 
\begin{equation*}
	R(\psi):=\{x\in X:d(T^nx,x)<\psi(n)\text{ for infinitely many }n\}
\end{equation*}
given a function $\psi:\mathbb{N}\to (0,\infty)$. Much has been done on the quantitative recurrence theory since then; for example, see \cite{BF,CWW,DFL,HLSW,KZ,KKP,Pe}. \par 
\smallskip
A topic closely related to recurrence theory is the so-called shrinking target problem, which is concerned with determining the speed at which the orbit of a $\mu$-generic point accumulates near a fixed point $y\in X$. More precisely, for $\psi:\mathbb{N}\to (0,\infty)$ and a given point $y\in X$, one can define the set 
\begin{equation*}
	A(\psi,y):=\{x\in X:d(T^nx,y)<\psi(n)\text{ for infinitely many }n\}.
\end{equation*}
There have been plenty of results concerning the zero-one laws for $\mu(A(\varphi,y))$ in specific systems; for example, see \cite{CK,FMP,HNPV,KM,LLVZ,Ph}. \par 
\smallskip
A more general setting called twisted recurrence, which can specialize into  shrinking target problem and quantitative recurrence, was introduced in \cite{KZ,LWW}. For $\psi: \mathbb{N}\to (0,\infty)$ and a Borel measurable function $f:X\to X$, one can consider the set 
\begin{equation*}
	R(\psi,f):=\{x\in X:d(T^nx,f(x))<\psi(n)\text{ for infinitely many }n\}.
\end{equation*}
 Clearly $R(\psi,f)=R(\psi)$ if $f$ is the identity function and $R(\psi,f)=A(\psi,y)$ if $f$ is the constant function with value $y$. When $f$ is Lipschitz, the zero-one laws for $\mu(R(\psi,f))$ were proved for some special systems in \cite{KZ,LWW}, including classical dynamical systems like $\beta$-transformations, Gauss transformations and left shift on Cantor sets. Specifically, in these cases 
 \begin{equation*}
 	\mu(R(\psi,f))=\begin{cases}
 		0 & \text{if }\sum_{n=1}^\infty \psi(n,x)^\delta <\infty \\
 		1 & \text{if }\sum_{n=1}^\infty \psi(n,x)^\delta =\infty,
 	\end{cases}
 \end{equation*} where $\delta$ is the Hausdorff dimension of the support of $\mu$.  \par 
\smallskip
In this paper, we prove quantitative Lipschitz recurrent properties of dynamical systems with exponential decay of correlations. To state the main result of the paper, we need to adapt and modify the settings and assumptions from \cite{KKP}. For the rest of the paper, let $X=[0,1]$ and $d$ be the standard metric, and assume for any sequence of positive real number $\{M_n\}_n$ contained in $(0,1)$, there exists a sequence of functions $\{r_n:X\to (0,1)\}_n$ such that $r_n(x)=\inf\{r:\mu(B(x,r))=M_n\}$ for all $x\in X$. For a function $f:X\to X$, the $f$-twisted recurrence set we are interested is defined by
 \begin{equation*}
 	R(\{M_n\}_n,f):=\{x\in X:T^nx\in B(f(x),r_n(x))\text{ for infinitely many }n\}.
 \end{equation*} We say that $\mu$ is \textsl{Ahlfors regular} if there exist constants $c,\delta>0$ such that 
 \begin{equation*}
 	\frac{1}{c}r^s\leq \mu(B(x,r))\leq cr^s \quad  \forall x\in \text{Supp}X \text{ and balls }B(x,r)\subset X
 \end{equation*} and that $\mu$ is \textsl{upper Ahlfors regular} if 
 there exist constants $c,\delta>0$ such that 
 \begin{equation}\label{uar}
 \mu(B(x,r))\leq cr^s \quad \forall x\in \text{Supp}X \text{ and balls }B(x,r)\subset X.
 \end{equation}
  If $\mu$ is Ahlfors regular, then $R(\{M_n\}_n,f)=R(\psi,f)$ where $\psi(n)=r_n(x)$. In general, these definitions are different. 
  Before stating our main theorems, we now specify the class of functions $f$ which we deal with by our technique. $f:X\to X$ is said to be \textsl{Lipschitz} if 
\begin{equation}
	\sup_{x,y\in X,x\neq y}\frac{d(f(x),f(y))}{d(x,y)}<\infty. \label{lip}
\end{equation}   \begin{definition}\label{def1}
 	Let $(X,\mu,T)$ be a measure-preserving system and $p:\mathbb{N}\to \mathbb{R}^+$ be a sequence. We say that the \textsl{correlations for the system decay as $p$ for $L^1$ against bounded variation (BV)}, if 
 	\begin{equation*}
	\left|\int (f\circ T^n)\cdot g\,d\mu-\int f\,d\mu\int g\,d\mu\right|\leq ||f||_{L^1}\cdot ||g||_{\text{BV}}\cdot p(n)
\end{equation*}
for all $n\in \mathbb{N}$ and for all functions $f$ with $||f||_1:=\int|f|\,d\mu<\infty $ and $g$ with $||g||_{\text{BV}}:=\sup_{x_i}\sum|g(x_{i+1})-g(x_i)|+\sup|g|<\infty$.
 \end{definition} 
 
\begin{remark}
 	Definition \ref{def1} is weaker than the uniform mixing condition found in \cite{KZ}, as for any balls $E,F\subset X$, we can take $f=\chi_E$ and $g=\chi_F$ and we will get 
 	\begin{equation*}
 		\left|\mu(T^{-n}E\cap F)-\mu(E)\mu(F)\right|\leq 3\mu(E) p(n).
 	\end{equation*}
 \end{remark}
 
\begin{remark}
 	For a non-increasing function $\psi:\mathbb{N}\to \mathbb{R}_{>0}$, we know that there exists some 
 	\begin{equation*}
 		\alpha\in \mathcal{W}(\psi):=\{\alpha \in [0,1]: |q\alpha-p|<\psi(q) \text{ for infinitely many natural numbers }p,q\}.
 	\end{equation*}
 	Then consider the system 
 	\begin{equation*}
 		T:[0,1]\to [0,1], \quad x\mapsto x+\alpha \, (\bmod{1})
 	\end{equation*} together with the Lebesgue measure. Then $R(\psi,\text{Id})=[0,1]$, where $\text{Id}:[0,1]\to [0,1]$ is the identity function. The convergence case of Theorem \ref{thm1} would fail; hence some mixing condition is need for a zero-one law for $R(\psi,f)$ to hold. For more details, see \cite[\S2]{KZ}.  
 \end{remark}
 \par 
 \smallskip
 We first state the sufficient condition for $R(\{M_n\}_n,f)=0$, which depends on the convergence of $\sum_{n=1}^\infty M_n$. If there exists a countable partition of subinterval $\{X_i\}_{i\in \mathcal{I}}$ so that $f|_{X_i}$ is Lipschitz for all $i\in \mathcal{I}$, then we say $f$ is \textsl{piecewise Lipschitz}. If there exists a countable partition of subinterval $\{X_i\}_{i\in \mathcal{I}}$ so that $f|_{X_i}$ is monotone for all $i\in \mathcal{I}$, then we say $f$ is \textsl{piecewise monotone}.
\begin{theorem}\label{thm1}
 	Let $p:\mathbb{N}\to \mathbb{R}^+$ be a function and assume the correlations for $(X,\mu,T)$ decay as $p$ for $L^1$ against BV with $\sum_{n=1}^\infty p(n)<\infty$. Let $\{M_n\}_n$ be a sequence contained in $(0,1)$. Suppose $f:X\to X$ is piecewise Lipschitz and piecewise monotone. If $\sum_{n=1}^\infty M_n<\infty$, then $\mu(R(\{M_n\}_n,f))=0$. 
 \end{theorem}
 \par 
 \smallskip
 For the full measure part, 
  \begin{theorem}\label{thm2}
 	Let $p:\mathbb{N}\to \mathbb{R}^+$, $(X,\mu,T)$, $\{M_n\}_n$ and $f:X\to X$ be as in Theorem \ref{thm1}. Additionally, suppose that 
\begin{itemize} 
	\item There exist $C>0$ and $0<\gamma<1$ such that $p(n)=C\gamma^n$.  
	\item $\mu$ is upper Ahlfors regular.
	\item For any $q>0$, 
	\begin{equation}\label{sum}
		\limsup_N \sum_{n=\lfloor q\log N\rfloor}^N M_n=\infty
	\end{equation}
\end{itemize}
 Then $\mu(R(\{M_n\}_n,f))=1$. 
 \end{theorem}\par 
 \smallskip
In \cite{KKP}, the authors proved Theorem \ref{thm1} and Theorem \ref{thm2} for $f$ being the identity map. We will apply some of the ideas and techniques used in the proof in \cite{KKP}, but the proof is different due to the new setup. \par 
 \smallskip
 We get an immediate corollary for Ahlfors regular systems.
 \begin{corollary}\label{cor1}
 	Let $p(n)=C\gamma^n$ for some $C>0$ and $0<\gamma<1$ and suppose the correlations for $(X,\mu,T)$ decays as $p$ for $L^1$ against BV. Suppose $\mu$ is $\delta$-Ahlfors regular.  Assume $f$ is piecewise Lipschitz and piecewise monotone. Then 
 	\begin{enumerate}
 		\item If $\sum_{n=1}^\infty \psi(n)^\delta<\infty$, then $R(\psi,f)$ is null. 
 		\item If for any $q>0$, 
		\begin{equation}
		\limsup_N \sum_{n=\lfloor q\log N\rfloor}^N M_n=\infty, \label{sum2}
		\end{equation} then $R(\psi,f)$ is full. 
 	\end{enumerate}
 \end{corollary} \par 
 \smallskip  
We shall remark that Corollary \ref{cor1} is a generalization of the convergence part of \cite[Theorem 1.2]{KZ} in the case $X=[0,1]$. Most notably, the expanding, bounded distortion and conformality assumptions are omitted. Corollary \ref{cor1} partially generalizes the divergence case of \cite[Theorem 1.2]{KZ}, with a stronger summability assumption of the measures of the targets \eqref{sum2}.\par 
 \smallskip
 The structure of the paper is as follows. In \S2, we reduce the proofs of Theorem \ref{thm1} and Theorem \ref{thm2} to the case where $f:X\to X$ is Lipschitz and it only changes monotonicity at finitely many points. In \S3 we construct a sequence of measurable sets whose limsup set is $R(\{M_n\}_n,f)$, and we estimate the measure of each set and conclude Theorem \ref{thm1}. In \S4, we study quasi-independence properties of this sequence and prove Theorem \ref{thm2}.
 
\section*{Acknowledgements}

The author would like to thank Tomas Persson for bringing this problem to his attention and discussing possible generalizations. The author is grateful to Dmitry Kleinbock for his wonderful advice and guidance throughout this project. 
 
 \section{Piecewise Lipschitz and piecewise monotone twists}
 
 By the properties of $R(\psi,f)$, to prove Theorem \ref{thm1} and Theorem \ref{thm2}, it suffices to show the statements for $f$ Lipschitz and monotone. 
 \begin{lemma}\label{lemm1}
 	Theorem \ref{thm1} and Theorem \ref{thm2} hold for $f$ Lipschitz and monotone. 
 \end{lemma}
 
 We will prove Lemma \ref{lemm1} in \S\S3-4. Here we first conclude Theorem \ref{thm1} and Theorem \ref{thm2} from this lemma. 
 
 \begin{proof}[Proofs of Theorem \ref{thm1} and Theorem \ref{thm2}]
 	Suppose there exists a countable collection of disjoint open intervals $\{X_i=(a_i,b_i)\}_{i\in \mathcal{I}}$ so that $\bigcup_{i\in \mathcal{I}}X_i$ is full and $f$ is Lipschitz and monotone on each $X_i$. Then for each $i\in \mathcal{I}$, $f$ is bounded on $X_i$ and hence we can define 
 	\begin{equation*}
 		f_i(x)=\begin{cases}
 			f(x) & \text{if }x\in (a_i,b_i)\\
 			\lim_{x\to a_i^-}f(x) & \text{if }x\leq a_i\\
 			\lim_{x\to b_i^+}f(x) & \text{if }x\geq b_i.
 		\end{cases}
 	\end{equation*}
 	Then $f_i$ is Lipschitz and monotone. By Lemma \ref{lemm1}, Theorem \ref{thm1} and Theorem \ref{thm2} hold for $f_i$ for each $i\in \mathcal{I}$. When $\mu(R(\psi,f_i))=0$ for all $i\in \mathcal{I}$, 
 	\begin{equation*}
 		\mu(R(\psi,f))=\sum_{i\in \mathcal{I}}\mu(R(\psi,f)\cap X_i)=\sum_{i\in \mathcal{I}}\mu(R(\psi,f_i)\cap X_i)=0;
 	\end{equation*} when $\mu(R(\psi,f_i))=1$ for all $i\in \mathcal{I}$, 
 	\begin{equation*}
 		\mu(R(\psi,f))=\sum_{i\in \mathcal{I}}\mu(R(\psi,f)\cap X_i)=\sum_{i\in \mathcal{I}}\mu(R(\psi,f_i)\cap X_i)=\sum_{i\in \mathcal{I}}\mu(X_i)=1,
 		\end{equation*}
 		so the theorems are proved once we prove Lemma \ref{lemm1} in the following sections. 
 \end{proof}

 \section{The convergence part}
 
In this section we prove the convergence part of Lemma \ref{lemm1}, thereby fixing $p:\mathbb{N}\to \mathbb{R}^+$, $(X,\mu,T)$, $\{M_n\}_n$ and $f:X\to X$  to be Lipschitz and monotone. Let us define 
\begin{equation*}
R_n(\{M_n\}_n,f):=\{x\in X:T^nx\in B(f(x),r_n(x))\}.
\end{equation*}
Without ambiguity, we shall denote $R_n(\{M_n\}_n,f)$ simply by $R_n$. Clearly $R(\{M_n\}_n,f)=\limsup_{n\to\infty }R_n$. \par 
\smallskip
We first prove a fact about the functions $r_n$.
\begin{lemma}\label{lemma3}
	For all $n\in \mathbb{N}$, $r_n$ is $1$-Lipschitz. 
\end{lemma}
\begin{proof}
	Let $x,y\in [0,1]$. Without the loss of generality, suppose $r_n(x)\leq r_n(y)$. Then 
	\begin{align*}
		B(x,r_n(x))\subset B(y,r_n(x)+d(x,y)),
	\end{align*} so $\mu(B(y,r_n(x)+d(x,y)))\geq M_n$ and hence
	\begin{align*}
		r_n(y)\leq r_n(x)+d(x,y),
	\end{align*}
	i.e., $|r_n(y)-r_n(x)|\leq |x-y|$. 
\end{proof}
\smallskip
For each $n$, we define $Y_n$ to be a subset of $[0,1]^2$ such that 
\begin{equation*}
	Y_n=\{(x,y):y\in B(f(x),r_n(f(x)))\}.
\end{equation*}
Then we have 
\begin{lemma}
	For each $n\in \mathbb{N}$, $Y_n$ is an open subset of $[0,1]^2$. 
\end{lemma}
\begin{proof}
	Fix $n\in \mathbb{N}$. We prove that $Y_n$ has closed complement in $[0,1]^2$. Let $\{(x_m,y_m)\}_m$ be a Cauchy sequence in the complement of $Y_n$ and we denote its limit in $[0,1]^2$ by $(x,y)$. We show that $(x,y)\not \in Y_n$. Let $\varepsilon>0$. Since $f$ is continuous, there exists $k\in \mathbb{N}$ so that 
	\begin{align*}
		|f(x_k)-f(x)|<\frac{\varepsilon}{L}\quad \text{and}\quad |y_k-y|<\varepsilon.
	\end{align*}
	Then 
	\begin{align*}
		|f(x)-y|\geq & |f(x_k)-y_k|-|f(x_k)-f(x)|-|y_k-y|\\
		 \geq & r_n(f(x_k))-2\varepsilon \\
		 \underset{\text{Lemma }\ref{lemma3}}\geq & r_n(f(x))-3\varepsilon.
	\end{align*}
	Since $\varepsilon$ is chosen arbitrarily, we must have $|f(x)-y|\geq r_n(f(x))$ as desired. 
\end{proof}\par 
\smallskip
Now we are ready to estimate the measure of $R_n$ for each $n$. 
\begin{lemma}\label{lemma5}
	For each $n\in \mathbb{N}$, 
	\begin{equation*}
		|\mu(R_n)-M_n|\leq 3p(n). 
	\end{equation*}
\end{lemma}
\begin{proof}
	Define $F_n:[0,1]^2\to \mathbb{R}$ to be the characteristic function of $Y_n$. Since $Y_n$ is open, we can approximate $F_n$ by the following a sequence of uniformly continuous functions $\{F_{n,k}\}_k$, where 
	\begin{align*}
		F_{n,k}(x,y):=\begin{cases}
			0 & \text{if }(x,y)\not\in Y_n\\
			\min\{1,kd((x,y),\partial Y_n) & \text{if }(x,y)\in Y_n
		\end{cases} 
	\end{align*}and $\partial Y_n$ denotes the boundary of $Y_n$. Note that $\{F_{n,k}\}_k$ is increasing in $k$ and it converges pointwise to $F_n$, so by the monotone convergence theorem, for each $\varepsilon>0$, there exists some $k$ so that \begin{equation*}
		\left|\int F_n(x,T^nx)\,d\mu(x)-F_{n,k}(x,T^nx)\,d\mu(x)\right|<\varepsilon
	\end{equation*} and 
	\begin{equation*}
		\left|\int F_n-\int F_{n,k}\right|<\varepsilon.
	\end{equation*} Since $F_{n,k}$ is $2k$-Lipschitz, we can choose a partition by intervals $\{I_h\}_{h=0}^{m-1}$ of $[0,1]$ so that 
\begin{equation*}
	\left|F_{n,k}(x,y)-\sum_{h=0}^{m-1}F_{n,k}(x,y_h)\chi_{I_h}(y)\right|<\varepsilon
\end{equation*} for all $x,y\in [0,1]$, where $y_h$ is the middle point of $I_h$. Then consider the integral
\begin{align}\label{eq1}
	\int \sum_{h=0}^{m-1}F_{n,k}(x,y_h)\chi_{I_h}(T^n x)
\end{align}
For each summand in \eqref{eq1}, apply the decay of correlations to get 
\begin{align*}
	\left|\int F_{n,k}(x,y_h)\chi_{I_h}(T^n x)\,d\mu(x)-\int F_{n,k}(x,y_h)\,d\mu(x)\int \chi_{I_h}(x)\,d\mu(x)\right|\\
	\leq \mu(I_h)||F_{n,k}(x,y_h)||_{BV}p(n)\\
\end{align*}
Note that for each $y_h$, 
\begin{equation*}
F_{n,k}(x,y_h)=\begin{cases}
	1 & \text{if }d(f(x),y_h)<r_n(f(x))\\
	0 & \text{else}
\end{cases}=\chi_{\{x:d(f(x),y_h)<r_n(f(x))\}},
\end{equation*}
so $||F_{n,k}(x,y_h)||_{BV}\leq 3$. Then summing over $h$ we get 
\begin{equation*}
	\left|\int \sum_{h=0}^{m-1}F_{n,k}(x,y_h)\chi_{I_h}(T^n x)\,d\mu(x)-\sum_{h=0}^{m-1} \int F_{n,k}(x,y_h)\,d\mu(x)\cdot \mu(I_h)\right|\leq 3p(n).
\end{equation*}
Hence 
\begin{align*}
	&\left|\mu(R_n)-M_n\right|\\
	=&\left|\int F_n(x,T^nx)\,d\mu(x)-\int F_n\,d\mu\,d\mu\right|\\
	\leq & \left|\int F_n(x,T^nx)\,d\mu(x)-\int F_{n,k}(x,T^nx)\,d\mu(x)\right| \\
	& +\left|\int F_{n,k}(x,T^nx)\,d\mu(x)-\sum_{h=0}^{m-1} \int F_{n,k}(x,y_h)\,d\mu(x)\cdot \mu(I_h)\right|\\
	&+ \left|\sum_{h=0}^{m-1} \int F_{n,k}(x,y_h)\,d\mu(x)\cdot \mu(I_h)-\int F_{n,k}\,d\mu\,d\mu\right|\\
	&+ \left|\int F_{n,k}\,d\mu\,d\mu-\int F_{n}\,d\mu\,d\mu\right|\\
	\leq & \varepsilon+(\varepsilon+3p(n))+\varepsilon+\varepsilon. 
\end{align*}
As we let $k\to \infty$, the lemma is proved. 
\end{proof}\par 
\smallskip
The convergence/zero measure case will follow immediately from Lemma \ref{lemma5}. 

\begin{proof}[Proof of Theorem \ref{thm1}]
	By Lemma \ref{lemma5}, $\sum_{n=1}^\infty\mu(R_n)\leq \sum_{n=1}^\infty (M_n+3p(n))<\infty$, so $\mu(R(\{M_n\}_n,f))=0$ by the Borel-Cantelli lemma. 
\end{proof}

\section{The divergence part}

In this section, we prove the divergence part of Lemma \ref{lemm2}, so throughout the section we will assume that $(X,\mu,T)$, $\{M_n\}_n$ and $f:X\to X$ satisfy the conditions stated in Theorem \ref{thm2} and that $\lim_{n\to\infty} p(n)=0$. The key ingredients for the proof of the divergence case are the estimate of the measure of each $R_n$, which was proved in Lemma \ref{lemma5}, and a quasi-independence property of $\{R_n\}_n$, which we will prove next. 

\begin{lemma}\label{lemma6}
There exist positive constants $K_1,K_2$ and $K_3$ such that for all $n,m\in \mathbb{N}$, 
	\begin{equation*}
		\mu(R_n\cap E_{R+m})\leq M_nM_{n+m}(1+K_1\sqrt{p(n)})+K_2(M_np(n)^{s/2}+M_{n+m}(p(n)^{s/2}+p(m)))+K_3p(n)^s
	\end{equation*} for some positive constants $K_1,K_2$ and $K_3$. 
\end{lemma}

\begin{proof}
	Let $n,m\in \mathbb{N}$. Define 
	\begin{equation*}
		E_{n,m}(x,y,z)=\begin{cases}
			1 & \text{if }y\in B(f(x),r_n(f(x))), z\in B(f(x),r_{n+m}(f(x)))\\
			0 & \text{otherwise}
		\end{cases}
	\end{equation*}
	Note that 
	\begin{equation*}
		\mu(R_n\cap R_{n+m})=\int F_{n,m}(x,T^nx,T^{n+m}x)\,d\mu(x)
	\end{equation*}
	We will approximate $E_{n,m}$ by a sum of products of characteristic functions of intervals. For each $\ell\in \mathbb{N}$, we first partition $X$ evenly into $\ell$ subintervals $\{I_h\}_{h=0}^{\ell-1}$ and denote the middle point of $I_h$ by $x_h$ for each $h=0,\ldots, \ell-1$. Define 
	\begin{equation*}
		F_{n,m,\ell}(x,y,z):=\sum_{h=0}^{\ell-1}\chi_{I_k}(x)\cdot \chi_{B(f(x_h),r_n(f(x_h))+\frac{L}{\ell})}(y)\cdot \chi_{B(f(x_h),r_{n+m}(f(x_h))+\frac{L}{\ell})}(z)
	\end{equation*}
	Note that for each $x\in I_h$ and $w\in B(f(x),r_n(f(x)))$, $|x-x_h|<\frac{1}{2\ell}$, $|f(x)-f(x_h)|<\frac{L}{2\ell}$ and $|r_n(f(x))-r_n(f(x_h))|<\frac{L}{2\ell}$, so $w\in B(f(x_h),r_n(f(x_h))+\frac{L}{\ell})$ and hence $F_{n,m,\ell}\geq F_{n,m}$. 
	Then 
	\begin{align*}
		&\int F_{n,m}(x,T^nx,T^{n+m}x)\,d\mu(x)\\
		\leq & \int F_{n,m,\ell}(x,T^nx,T^{n+m}x)\,d\mu(x)\\
		=& \sum_{h=0}^{\ell-1} \int \chi_{I_k}(x)\cdot \chi_{B(f(x_h),r_n(f(x_h))+\frac{L}{\ell})}(T^nx)\cdot \chi_{B(f(x_h),r_{n+m}(f(x_h))+\frac{L}{\ell})}(T^{n+m}x)\,d\mu(x)\\
		\underset{\text{BV}}\leq & \sum_{h=0}^{\ell-1} \int \chi_{B(f(x_h),r_n(f(x_h))+\frac{L}{\ell})}(x)\cdot \chi_{B(f(x_h),r_{n+m}(f(x_h))+\frac{L}{\ell})}(T^{m}x)\,d\mu(x)\cdot \left(\mu(I_h)+3p(n)\right)\\
		\underset{\text{BV}}\leq & \sum_{h=0}^{\ell-1} \left(\mu\left(B\left(f(x_h),r_n(f(x_h))+\frac{L}{\ell}\right)\right)+3p(m)\right)\left(\mu\left(B\left(f(x_h),r_{n+m}(f(x_h))+\frac{L}{\ell}\right)\right)\right)\\
		& \cdot (\mu(I_k)+3p(n))
	\end{align*}
	By upper Ahlfors regularity \eqref{ar}, 
	\begin{equation*}
		\mu\left(B\left(f(x_h),r_n(f(x_h))+\frac{L}{\ell}\right)\right)\leq M_n+2c\left(\frac{L}{\ell}\right)^s 
	\end{equation*} and 
	\begin{equation*}
		\mu\left(B\left(f(x_h),r_{n+m}(f(x_h))+\frac{L}{\ell}\right)\right)\leq M_{n+m}+2c\left(\frac{L}{\ell}\right)^s.
	\end{equation*}
	Now we pick $\ell$ so that $\sqrt{p(n)}/2<\frac{L}{\ell}<\sqrt{p(n)}$. Then 
	\begin{align*}
		&\int F_{n,m}(x,T^nx,T^{n+m}x)\,d\mu(x)\\
		\leq & (M_n+2cp(n)^{s/2}+3p(m))(M_{n+m}+2cp(n)^{s/2})\cdot (1+6L\sqrt{p(n)})\\
		\leq & M_nM_{n+m}(1+K_1\sqrt{p(n)})+K_2(M_np(n)^{s/2}+M_{n+m}(p(n)^{s/2}+p(m)))+K_3p(n)^s
	\end{align*}
	where $K_1=6L$, $K_2=(1+6L\sqrt{\sup_np(n)})(2c+3)$ and $K_3=4c^2(1+6L\sqrt{\sup_np(n)})$. 
\end{proof}\par 
\smallskip

To prove the divergence case, we will need the Chung-Erd\"{o}s inequality. 

\begin{lemma}
	In a probability space, for measurable sets $A_1,\ldots,A_n$, 
	\begin{equation*}
		\mu(A_1\cup \cdots\cup A_n)\geq \frac{\left(\sum_{j=1}^n \mu(A_j)\right)^2}{\sum_{j,k=1}^n \mu(A_j\cup A_k)}
	\end{equation*}
\end{lemma}\par 
\smallskip
Now we finish the proof of Theorem \ref{thm2}. 

\begin{proof}[Proof of Theorem \ref{thm2}]
	In addition to the assumptions at the beginning of this section, we assume that $p(n)=C\gamma^{n}$ for some $C>0$ and $0<\gamma<1$. Let 
	\begin{equation*}
		I_N=\left\{j:-\frac{2}{s}\log_\gamma N\leq j\leq N\right\}
	\end{equation*} and 
	\begin{equation*}
		U_N=\bigcup_{j\in I_N}R_j.
	\end{equation*}
	Note that $\limsup_n R_n=\limsup_N U_N$.   \par 
	Let 
	\begin{equation*}
		S_N=\sum_{j\in I_N}\mu(R_j)
	\end{equation*} and 
	\begin{equation*}
		\sigma_N=\sum_{j\in I_N}M_j
	\end{equation*}
	By Lemma \ref{lemma5}, we have 
	\begin{equation*}
		\sigma_N-c_1\leq S_N\leq \sigma_N+c_1
	\end{equation*}
	where the constant $c_1=C\gamma^{-1}$. On the other hand, let 
	\begin{equation*}
		C_N=\sum_{j,k\in I_N}\mu(R_j\cap R_k)
	\end{equation*}
	and by Lemma \ref{lemma6}, we have 
	\begin{align*}
		C_N=&S_N+2\sum_{j,k\in I_N, j>k}\mu(E_j\cap E_k)\\
		\leq & S_N+(1+K_1CN^{-1/s})\sigma_N^2+2K\sum_{j,k\in I_N,j>k}(M_k\gamma^{ks/2}+M_j(\gamma^{ks/2}+\gamma^{j-k})+\gamma^{ks/2})
	\end{align*} where $K=K_2+K_3$. 
	Denote  
	\begin{equation*}
		D_N:=\sum_{j,k\in I_N,j>k}(M_k\gamma^{ks/2}+M_j(\gamma^{ks/2}+\gamma^{j-k})+\gamma^{ks/2})
	\end{equation*}
	We show that $R_N$ is bounded, proceeding term by term. 
	Each of the first two and last sums is less than or equal to 
	\begin{align*}
		\sum_{j=-\frac{2}{s}\log_\gamma N}^N \sum_{k=-\frac{2}{s}\log_\gamma N}^{j-1}\gamma^{ks/2}\leq \sum_{j=-\frac{2}{s}\log_\gamma N}^N c_2 e^{-\log N}\leq c_2
	\end{align*} for some constant $c_2$ dependent on the bound in  \eqref{sum}. 
	For the third sum, 
	\begin{equation*}
		\sum_{j,k\in I_N, j>k}M_j \gamma^{j-k}=\sum_{j=-\frac{2}{s}\log_\gamma N} M_j \sum_{k=-\frac{2}{s}\log_\gamma N}^{j-1} \gamma^{j-k}\leq c_3 \sigma_N
	\end{equation*}
	Hence
	\begin{align*}
		C_N\leq& S_N+(1+K_1CN^{-1/2})\sigma_N^2+2K(3c_2+c_3\sigma_N)\\
		\leq & \sigma_N+c_1+(1+K_1CN^{-1/2})\sigma_N^2+2K(3c_2+c_3\sigma_N)
	\end{align*} 
	Now we can use the Chung-Erd\"{o}s inequality to conclude that 
	\begin{equation}\label{cei}
		\mu(U_N)\geq \frac{S_N^2}{C_N}\geq \frac{(\sigma_N-c_1)^2}{(1+K_1CN^{-1/2})\sigma_N^2+c_4\sigma_N+c_5}
	\end{equation}
	for some constants $c_4,c_5$; as we take $\limsup_{N\to \infty}$ in  \eqref{cei}, we have that 
	\begin{equation*}
		\limsup_N \mu(U_N)\geq 1
	\end{equation*}
	Hence we proved that $\limsup_n R_n=\limsup_N U_N=1$. 
\end{proof}

\end{document}